\documentclass{amsart}
\usepackage[foot]{amsaddr}
\usepackage{graphicx}

\usepackage{amsfonts}
\usepackage{amsmath}
\usepackage{amssymb}
\usepackage{amsthm} 
\usepackage{array}
\usepackage{float}
\usepackage[toc,page,titletoc]{appendix}
\usepackage{longtable}
\usepackage{graphicx}
\usepackage[font=singlespacing]{caption}
\usepackage[normalem]{ulem}
\usepackage{nicefrac}
\usepackage{units}
\usepackage{rotating}
\usepackage{afterpage}
\usepackage{tikz}
\usepackage{tikz-cd}
\usetikzlibrary{arrows,positioning}

\usepackage{mathrsfs}
\usepackage{amscd}

\usepackage[colorlinks,citecolor=blue,urlcolor=blue,bookmarks=false,hypertexnames=true]{hyperref}

\usepackage{cleveref}
\usepackage{comment}
\usepackage{upgreek}
\usepackage{ragged2e}

\usepackage[mathscr]{euscript}
\usepackage{scalerel,stackengine}

\usepackage{enumitem}

\newtheorem{theorem}{Theorem}[section]
\newtheorem{lemma}[theorem]{Lemma}

\theoremstyle{definition}

\theoremstyle{remark}

\numberwithin{equation}{section}

\theoremstyle{definition}

\newcommand{\g}{\mathnormal{g}}
\newcommand{\overbar}[1]{\mkern 1.5mu\overline{\mkern-1.5mu#1\mkern-1.5mu}\mkern 1.5mu}

\newcommand{\ms}[1]{\mathscr{#1}}

\stackMath
\newcommand\reallywidehat[1]{%
\savestack{\tmpbox}{\stretchto{%
  \scaleto{%
    \scalerel*[\widthof{\ensuremath{#1}}]{\kern-.6pt\bigwedge\kern-.6pt}%
    {\rule[-\textheight/2]{1ex}{\textheight}}
  }{\textheight}%
}{0.5ex}}%
\stackon[1pt]{#1}{\tmpbox}%
}

\newcounter{parno}[paragraph]

\newcounter{subparno}[section]

\begin{document}

\title{On Automorphism Groups of Hardy Algebras}

\author{Rene Ardila}




\dedicatory{This is a pre-print of an article published in the Annals of Functional Analysis. The final authenticated version is available online at \url{https://doi.org/10.1007/s43034-020-00079-5}.}

\begin{abstract}
Let $E$ be a $W^{*}$-correspondence and let $H^{\infty}(E)$ be the associated Hardy algebra. The unit disc of intertwiners $\mathbb{D}((E^{\sigma})^{*})$ plays a central role in the study of $H^{\infty}(E)$.  We show a number of results related to the automorphism groups of both  $H^{\infty}(E)$ and $\mathbb{D}((E^{\sigma})^{*})$. We find a matrix representation for these groups and describe  several features of their algebraic structure. Furthermore, we show an application of $Aut(\mathbb{D}({(E^{\sigma}})^*))$ to the study of Morita equivalence of $W^{*}$-correspondences.

\end{abstract}

\maketitle




\section{Introduction} \label{Introduction}
This note contributes to a circle of ideas developed by Muhly and Solel related to automorphisms of the Hardy algebra $H^{\infty}(E)$. In \cite{Muhly2008b}, they showed that the automorphisms of  $H^{\infty}(E)$ are obtained by composition with certain biholomorphic automorphisms of $\mathbb{D}((E^{\sigma})^{*})$, the open unit disc of the intertwining space. As a result, we can view the automorphism group of $H^{\infty}(E)$ as a subgroup of the automorphism group of $\mathbb{D}({(E^{\sigma}})^*)$. That is, $Aut(H^{\infty}(E)) \leq Aut(\mathbb{D}({(E^{\sigma}})^*))$. We answer a number of questions related to the algebraic structure of both groups as well as the relationship between them. We find a matrix representation for both groups and we find a third group containing both automorphism groups. In the last section, we show how $Aut(\mathbb{D}({(E^{\sigma}})^*))$ can be used to develop a categorical approach to Morita equivalence of $W^{*}$-correspondences.

\section{Preliminaries} \label{preliminaries}
For any Hilbert spaces $H$ and $K$, let $B(H,K)$ denote the Banach space of all bounded linear operators from $H$ to $K$ with the operator norm. A \emph{$J^{*}$-algebra} is a closed complex-linear subspace $\mathcal{U}$ of $B(H,K)$ such that $AA^{*}A \in \mathcal{U}$ whenever $A \in \mathcal{U}$. A right $C^{*}$-module $E$ over a $C^{*}$-algebra $A$ is said to be \emph{selfdual} if every continuous $A$-module map $f:E\to A$ is of the form $f(\cdot)=\langle y, \cdot\rangle$, for some $y\in E$.  We say that $E$ is a right \emph{$W^{*}$-module} if $E$ is a selfdual right $C^{*}$-module over a $W^{*}$-algebra. We write $\ms{L}_{A}(E)$ (or simply $\ms{L}(E)$) for the space of \emph{adjointable} $A$-module maps on $E$. (Recall that Paschke showed that if $E$ is a $W^{*}$-module, the set of adjointable $A$-module maps on $E$ is the $W^{*}$-algebra of bounded $A$-module maps on $E$ \cite[Corollary 3.5 and Proposition 3.10]{Paschke1973}). A \emph{$W^{*}$-correspondence} $(E,A)$ is a right $W^{*}$-module $E$ over a $W^{*}$-algebra $A$ for which there exists a unital normal $*$-homomorphism $\varphi$: $A \to \ms{L}(E)$. The center of a $W^{*}$-correspondence $(E,A)$ is the set $\mathfrak{Z}(E)=\{x\in E : a\cdot x=x\cdot a\text{ for all } a\in A\}$. 

The $W^{*}$-module tensor product $\overbar{\otimes}_{A}$ is defined to be the selfdual completion (the weak$^{*}$-completion) of the balanced $C^{*}$-module interior tensor product. The \emph{Fock space} $\ms{F}(E)$ of $E$ is defined as the ultraweak direct sum of all the tensor powers of $E$. That is, $\ms{F}(E):=\bigoplus^{wc}_{n\in \mathbb{N}_{0}}E^{\overbar{\otimes}n}$. This space is itself a a $W^{*}$-correspondence over $A$. The left action of $A$ on $\ms{F}(E)$ is given by the map $\varphi_{\infty}$ defined by $\varphi_{\infty}(a)=$ diag$(a,\varphi(a),\varphi ^{(2)}(a) ,\varphi ^{(3)}(a)   , \cdots)$ where $\varphi^{(n)}(a)(x_{1} \otimes x_{2}\otimes \cdots\otimes x_{n})=(\varphi(a)x_{1}) \otimes x_{2}\otimes \cdots\otimes x_{n} \in E^{\overbar{\otimes}n}$. Given $x\in E$, the \emph{creation operator} $T_{x}\in \ms{L}(\ms{F}(E))$ is defined by $T_{x}(\eta)=x\otimes \eta$, $\eta \in \ms{F}(E))$. The \emph{tensor algebra} over $E$, denoted $\mathcal{T}_{+}(E)$ is defined to be the norm closed subalgebra of $\ms{L}(\ms{F}(E))$ generated by $\varphi_{\infty}(A)$ and $\{T_{x}:x \in E \}$. The ultraweak closure of $\mathcal{T}_{+}(E)$ in $\ms{L}(\ms{F}(E))$ is called the \emph{Hardy Algebra} of $E$, and is denoted by $H^{\infty}(E)$. When $E=A=\mathbb{C}$, $H^{\infty}(E)$ is the classical Hardy space $H^{\infty}(\mathbb{T})$. That is, $H^{\infty}(E)$ is a noncommutative generalization of the classic Hardy algebra $H^{\infty}(\mathbb{T})$ of bounded analytic functions on the open unit disc. When $A= \mathbb{C}$ and $E=\mathbb{C}^{n}$, $H^{\infty}(E)$ is Popescu's noncommutative Hardy space $\ms{F}^{\infty}$ \cite{Popescu1991} and the noncommutative analytic Toeplitz algebra studied by Davidson and Pitts \cite{DPit98a,DPit98b}. More information about $H^{\infty}(E)$ can be found in \cite{Muhly2004a,Muhly2008b,Muhly2011a,Muhly2009}.

A \emph{completely contractive covariant representation} of $(E,A)$ is a pair $(T,\sigma)$ where $\sigma:A\to B(H)$ is a normal $*$-representation of $A$ and $T:E\to B(H)$ is a linear, completely contractive $w^{*}$-continuous representation of $E$ satisfying $T(axb)=\sigma(a)T(x)\sigma(b)$ for all $x\in E$ and $a,b\in A$. As shown in \cite{Muhly2004a}, the linear map $\widetilde{T}$ defined on the algebraic tensor product $E\otimes H$ by $\widetilde{T}(x\otimes h)=T(x)h$ extends to an operator of norm at most 1 on the completion $E\overbar{\otimes}_{\sigma}H$. The bimodule property of $T$ is equivalent to the equation $\widetilde{T}(\sigma^{E}\circ \varphi(a))= \widetilde{T}(\varphi(a)\otimes I)=\sigma(a)\widetilde{T}$ for all $a\in A$, which means that $\widetilde{T}$ intertwines the representations $\sigma$ and $\sigma^{E}\circ \varphi$ of $A$ on $H$ and $E\otimes H$ respectively. The space composed of all these intertwiners is called the \emph{intertwining space}, and it is usually denoted as $\mathcal{I}(\sigma ^{E} \circ \varphi, \sigma)$ or $(E^{\sigma})^{*}$. The unit ball of the intertwining space is then denoted by $\mathbb{D}(({E^{\sigma}})^*)$ or just $\mathbb{D}({E^{\sigma}}^*)$. The elements of $\mathbb{D}(({E^{\sigma}})^*)$ determine ultraweakly continuous representations of $H^{\infty}(E)$  \cite[Corollary 2.14]{Muhly2004a}.

As shown throughout the work of Muhly and Solel on Hardy algebras (for example in \cite[Remark 2.14]{Muhly2008b}), we can view the elements of $H^{\infty}(E)$ as $B(H)$-valued functions defined on $\mathbb{D}(({E^{\sigma}})^*)$. That is, we view $H^{\infty}(E)$ as an algebra of functions on its representation space. If $\eta^{*} \in \mathbb{D}(({E^{\sigma}})^*)$ and $X$ is an element of $H^{\infty}(E)$, the function $\widehat{X}$ is defined by $\widehat{X}(\eta^{*})=\sigma \times \eta^{*}(X)$, where $\sigma \times \eta^{*}$ is the ultraweakly continuous completely contractive representation of $H^{\infty}(E)$ determined by $\sigma$ and $\eta^{*}$. 

The study of Hardy algebras and unit balls of intertwiners also offers a unique perspective of noncommutative function theory. As shown in \cite{Muhly2013}, the family $\{\mathbb{D}({E^{\sigma}}^*)\}_{\sigma \in NRep(A)}$ satisfies several properties which are similar to the properties of the domains considered by J.L. Taylor in \cite[section 6]{Tay72c}, the fully matricial sets of Voiculescu \cite{Voi2005, Voi2010} and the noncommutative sets studied by Helton-Klep-McCullough \cite{HKM2011b,HKM2011a,HKMS2009}, Kaliuzhnyi-Verbovetskyi and Vinnikov \cite{Kaliuzhnyi2012,kaliuzhnyi2014foundations,K-VV2009}.

\section{$Aut(\mathbb{D}({(E^{\sigma}})^*))$ and $Aut(H^{\infty}(E))$}

In this section, building on the work of Muhly and Solel in \cite{Muhly2008b}, we give new results about $Aut(\mathbb{D}({E^{\sigma}}^*))$ and its relationship to $Aut(H^{\infty}(E))$, the automorphism group of the Hardy algebra of $E$.

Let $(E,A)$ be a $W^{*}$-correspondence. Let $Aut(H^{\infty}(E))$ denote the completely isometric, $w^{*}$-homeomorphic automorphisms of $H^{\infty}(E)$ fixing $\varphi(A)$ elementwise.  In \cite[Lemma 4.20 and Theorem 4.21]{Muhly2008b}, Muhly and Solel showed that each automorphism $\alpha$ in $Aut(H^{\infty}(E))$ is obtained by composition with some element in $Aut(\mathbb{D}(E^{\sigma})^{*})=\{g:\mathbb{D}(E^{\sigma})^{*}\to \mathbb{D}(E^{\sigma})^{*}$ biholomorphic$\}$ in the following way: $\widehat{\alpha(X)}(\eta^{*})=\widehat{X}(\g(\eta^{*}))$. Furthermore, if an element $g$ in  $Aut(\mathbb{D}(E^{\sigma})^{*})$ implements an automorphism $\alpha$ in $Aut(H^{\infty}(E))$ then $g$ preserves $\mathbb{D}\mathfrak{Z}((E^{\sigma})^{*})$ \cite[Theorems 4.9 and 4.21]{Muhly2008b}. With this in mind, we can think of $Aut(H^{\infty}(E))$ as a subgroup of $Aut(\mathbb{D}(E^{\sigma})^{*})$. That is, $Aut(H^{\infty}(E))\leq Aut(\mathbb{D}(E^{\sigma})^{*})$.


Since $(E^{\sigma})^{*}$ is a $J^{*}$ algebra 
, the biholomorphic automorphisms $\g$ in $ Aut(\mathbb{D}({E^{\sigma}}^*))$ are of the form $\g= \omega  \circ \g_{\gamma}$ \cite[Definition 1 and Theorem 3]{Harris1974}, where $\omega$ is a surjective linear isometry on $(E^{\sigma})^{*}$, $\gamma\in \mathbb{D}({E^{\sigma}})$ and $g_{\gamma}$ is a  M\"{o}bius transformation of $\mathbb{D}(E^{\sigma})^{*}$ given by $\g _{\gamma}(\eta^{*})=\Delta_{\gamma}(I_{H}-\eta^{*}\gamma)^{-1}(\gamma^{*}-\eta^{*})  \Delta_{\gamma ^{*}}^{-1}$, where $\Delta_{\gamma}=(I_{H}- \gamma^{*}\gamma)^{1/2}$ and $\Delta_{\gamma ^{*}}=(I_{E \otimes H}-\gamma\gamma^{*})^{1/2}$ \cite[equation (25)]{Muhly2008b}. Each M\"{o}bius map $\g _{\gamma}$ is a biholomorphic automorphism of $Aut(\mathbb{D}({E^{\sigma}}^*))$ mapping $\gamma^{*}$ to $0$ and $0$ to $\gamma^{*}$, thus satisfying $\g _{\gamma}^{2}=id$. Therefore, for every pair of intertwiners $\eta^{*}_{1}, \eta^{*}_{2}\in \mathbb{D}({E^{\sigma}}^*)$, the map $\g _{\eta_{2}}\circ\g _{\eta_{1}}$ takes $\eta^{*}_{1}$ to $\eta^{*}_{2}$. That is, $\mathbb{D}({E^{\sigma}}^*)$ is a homogeneous domain.
 
\begin{lemma}\label{uniquedecomposition}
Let $g\in Aut(\mathbb{D}({E^{\sigma}}^*))$. Then $g=\omega\circ g_{g^{-1}(0)^{*}}$. Furthermore, this decomposition is unique. That is, if $g=\omega\circ g_{g^{-1}(0)^{*}}=\omega '\circ g_{\gamma}$, then $\omega=\omega '$ and $\gamma =g^{-1}(0)^{*}$.
\end{lemma}
\begin{proof}
$g=\omega\circ g_{\gamma}$. So $g(\gamma)=\omega\circ g_{\gamma}(\gamma)=\omega(0)=0$. So $\gamma=g^{-1}(0)$. Now suppose there are $\omega'$ and $g_{\gamma}$ such that $\omega\circ g_{g^{-1}(0)^{*}}=g=\omega'\circ g_{\gamma}$. Then $\omega'=\omega\circ g_{g^{-1}(0)^{*}} \circ g_{\gamma}$. Since $g_{g^{-1}(0)^{*}} \circ g_{\gamma}=\omega'' \circ g_{(g_{g^{-1}(0)^{*}} \circ g_{\gamma})^{-1}(0)}$ for some linear isometry $\omega''$, we have $\omega'=\omega \circ \omega'' \circ g_{(g_{g^{-1}(0)^{*}} \circ g_{\gamma})^{-1}(0)}=W\circ g_{g_{\gamma}(g^{-1}(0))^{*}}$, since $\omega \circ \omega''$ is an isometry $W$ and $(g_{g^{-1}(0)^{*}}\circ g_{\gamma}) ^{-1}=g_{\gamma}(g_{g^{-1}(0)^{*}}(0))=g_{\gamma}(g^{-1}(0))$. Thus $g_{g_{\gamma}(g^{-1}(0))^{*}}=W^{-1}\circ \omega'$ is a linear isometry. So $0=g_{g_{\gamma}(g^{-1}(0))^{*}}(0)=g_{\gamma}(g^{-1}(0)) $. So $\gamma=g^{-1}(0)^{*}$ and $\omega'=\omega\circ g_{\gamma}^{2}=\omega$.
\end{proof}
As usual, we denote the center of a group $N$ by $Z(N)$.
\begin{theorem}\label{center}
$Z(Aut(\mathbb{D}({E^{\sigma}}^*))=\{id_{Aut(\mathbb{D}({E^{\sigma}}^*))} \}$ and  $Z(Aut(H^{\infty}(E))=\{id_{Aut(H^{\infty}(E))} \}$.
\end{theorem}
\begin{proof}
First, note that $g_{0}=-id$. Clearly, $id \in Z( Aut(\mathbb{D}({E^{\sigma}}^*))$. Let $g=\omega\circ g_{g^{-1}(0)^{*}} \in Aut(\mathbb{D}({E^{\sigma}}^*))$. If $g^{-1}(0)\neq 0$ then $g_{0}\circ g(g^{-1}(0))=g_{0}(0)=0=\omega(0)\neq \omega\circ g_{g^{-1}(0)^{*}}(-g^{-1}(0))= \omega\circ g_{g^{-1}(0)^{*}}\circ g_{0}(g^{-1}(0))=g\circ g_{0}(g^{-1}(0))$. Thus, if $g\in Z( Aut(\mathbb{D}({E^{\sigma}}^*))$, we must have $g^{-1}(0)=0$. So $g(0)=0$. Assume $g\neq id$. Then there is $\gamma^{*}\in\mathbb{D}({E^{\sigma}}^*)$ such that $g(\gamma^{*})\neq \gamma^{*}$. Then  $g_{\gamma}\circ g(\gamma^{*})\neq g_{\gamma}(\gamma^{*})=0=g(0)=g\circ g_{\gamma}(\gamma^{*})$. So $g\notin Z(Aut(\mathbb{D}({E^{\sigma}}^*))$ and $Z(Aut(\mathbb{D}({E^{\sigma}}^*))=\{id\}$.

Now let $\alpha'\in Aut(H^{\infty}(E))$ and suppose $\alpha'$ commutes with every $\alpha\in Aut(H^{\infty}(E))$. Let $X\in H^{\infty}(E)$. $\widehat{\alpha(X)}(\eta^{*})=\widehat{X}(\g(\eta^{*}))$ and $\widehat{\alpha'(X)}(\eta^{*})=\widehat{X}(\g'(\eta^{*}))$. Then $\reallywidehat{\alpha\circ \alpha'(X)}(\eta^{*})=\reallywidehat{\alpha (\alpha'(X))}(\eta^{*})=\widehat{\alpha'(X)}(g(\eta^{*}))=\widehat{X}(g'\circ g(\eta^{*}))$ and $\reallywidehat{\alpha'\circ \alpha(X)}(\eta^{*})=\reallywidehat{\alpha' (\alpha(X))}(\eta^{*})=\widehat{\alpha(X)}(g'(\eta^{*}))= \widehat{X}(g\circ g'(\eta^{*}))$. Thus $g'\in Z(Aut(\mathbb{D}({E^{\sigma}}^*))$. So $g'=id_{Aut(\mathbb{D}({E^{\sigma}}^*))}$ and $Z(Aut(H^{\infty}(E))=\{id_{Aut(H^{\infty}(E))}\}$.
\end{proof}
Note that if $(E,A)$ is a $W^{*}$-graph correspondence and $g\in Aut(\mathbb{D}({E^{\sigma}}^*))$, then $g\in Aut(H^{\infty}(E))$ if and only if for each intertwiner $\eta^{*}=(T_{ij})\in \mathfrak{Z}((E^{\sigma})^{*})$, the zero blocks of $g(\eta^{*})$ are the same zero blocks of $\eta^{*}$ and the non-zero blocks of $g(\eta^{*})$ are multiples of identities. This follows from \cite[Corollary 4.2]{Ardila2019} and the fact that the elements in $ Aut(H^{\infty}(E))$ preserve $\mathfrak{Z}((E^{\sigma})^{*})$.


Our next goal is to give a matrix representation of $Aut(\mathbb{D}({E^{\sigma}}^*))$. With this in mind, we construct a set $\mathbb{P}$, which will allow us to express the elements of $Aut(\mathbb{D}({E^{\sigma}}^*))$ as matrices acting on $\mathbb{P}$ by right matrix multiplication.
Let $\mathbb{P}=\{(U, \eta ^{*})\} / \sim$ where $U$ is an invertible operator in $\sigma(A)'$
, $\eta ^{*} \in (E^{\sigma})^*$ and $(U_{1}, \eta_{1} ^{*}) \sim (U_{2}, \eta_{2} ^{*})$ if there is an invertible operator $C \in \sigma(A)'$ such that $(CU_{1}, C\eta_{1} ^{*}) = (U_{2}, \eta_{2} ^{*})$. The role of $\mathbb{P}$ in our analysis will be similar to the role played by the complex projective line in the study of Mobius transformations of the complex plane. So  we can think of the elements in the set 
\begin{equation*}
\{(U,U \eta ^{*})\quad|\quad U\text{ invertible in } \sigma(A)'\}
\end{equation*}
as ``homogeneous coordinates" of $\eta ^{*}$.

In particular, each $ \eta ^{*} \in (E^{\sigma})^*$ has ``homogeneous coordinates" $(I_{H}, \eta ^{*})$. So each $ \eta ^{*} \in (E^{\sigma})^*$ can be identified with the equivalence class $[(I_{H}, \eta ^{*})]$, which we will also denote by $[I_{H} \quad \eta ^{*}]$, so that we can view it as both an equivalence class and a 1 by 2 block matrix.

By \cite[Theorem 4.21(ii)]{Muhly2008b}, if $\g\in Aut(\mathbb{D}({E^{\sigma}}^*))$ implements $\alpha \in Aut(H^{\infty}(E))$, then there is a $\gamma^{*} \in \mathbb{D}\mathfrak{Z}(({E^{\sigma}}^*) $ and a unitary operator $u$ in $\ms{L}(E)$ such that $u(\mathfrak{Z}(E))=\mathfrak{Z}(E)$ and such that $\g(\eta^{*})=g_{\gamma}(\eta^{*})(u\otimes I_{H})$. The following theorem shows that we can represent $\g$ by the matrix 
$T=\left( 
\begin{matrix} 
\Delta_{\gamma}^{-1} & \gamma^{*} \Delta_{\gamma ^{*}}^{-1}(u\otimes I_{H}) \\
-\gamma \Delta_{\gamma}^{-1} & -\Delta_{\gamma ^{*}}^{-1}(u\otimes I_{H})
\end{matrix}
\right)$
 acting on $[I_{H} \quad \eta ^{*}]\in \mathbb{P}$ by right matrix multiplication.
\begin{theorem}
If $g\in Aut(\mathbb{D}({E^{\sigma}}^*))$ implements $\alpha\in Aut(H^{\infty}(E))$, so that $\g(\eta^{*})=g_{\gamma}(\eta^{*})(u\otimes I_{H})$, then
\begin{equation*}
[I_{H} \quad \eta ^{*}]
\left( 
\begin{matrix} 
\Delta_{\gamma}^{-1} & \gamma^{*} \Delta_{\gamma ^{*}}^{-1}(u\otimes I_{H}) \\
-\gamma \Delta_{\gamma}^{-1} & -\Delta_{\gamma ^{*}}^{-1}(u\otimes I_{H})
\end{matrix}
\right)
=[I_{H} \quad \g(\eta ^{*})]
\end{equation*}
\end{theorem}
\begin{proof}

\begin{equation*}
\begin{split}
[I_{H} \quad \eta ^{*}] &
\left( 
\begin{matrix} 
\Delta_{\gamma}^{-1} & \gamma^{*} \Delta_{\gamma ^{*}}^{-1}(u\otimes I_{H}) \\
-\gamma \Delta_{\gamma}^{-1} & -\Delta_{\gamma ^{*}}^{-1}(u\otimes I_{H})
\end{matrix}
\right)\\
&=[\Delta_{\gamma}^{-1}-\eta ^{*}\gamma \Delta_{\gamma}^{-1} \qquad (\gamma^{*} \Delta_{\gamma ^{*}}^{-1}-\eta ^{*}\Delta_{\gamma ^{*}}^{-1})(u\otimes I_{H})] \\
&= [I_{H} \qquad (\Delta_{\gamma}^{-1}-\eta ^{*}\gamma \Delta_{\gamma}^{-1})^{-1} (\gamma^{*} \Delta_{\gamma ^{*}}^{-1}-\eta ^{*}\Delta_{\gamma ^{*}}^{-1})(u\otimes I_{H})]\\
&= [I_{H} \qquad ((I_{H}-\eta ^{*}\gamma) \Delta_{\gamma}^{-1})^{-1} (\gamma^{*}-\eta ^{*})\Delta_{\gamma ^{*}}^{-1}(u\otimes I_{H})]\\
&= [I_{H} \qquad \Delta_{\gamma}(I_{H}-\eta ^{*}\gamma)^{-1} (\gamma^{*}-\eta ^{*})\Delta_{\gamma ^{*}}^{-1}(u\otimes I_{H})]\\
&= [I_{H} \qquad g_{\gamma}(\eta ^{*})(u\otimes I_{H})]\\
&= [I_{H} \qquad g(\eta ^{*})]
\end{split}
\end{equation*}
We check that the right side of the first equality is an element of $\mathbb{P}$. Let $a\in A$ and $\eta_{1} ^{*}, \eta_{2}^{*}, \eta_{3}^{*}\in(E^{\sigma})^*$. Then $\eta_{1} ^{*}\eta_{2}\sigma(a)=\eta_{1} ^{*}(\varphi(a)\otimes I_{H})\eta_{2}=\sigma(a)\eta_{1} ^{*}\eta_{2}$. So $\eta_{1} ^{*}\eta_{2}\in \sigma(A)'$. Then, since we are assuming $\gamma \in \mathbb{D}({E^{\sigma}}^*))$, $\Delta_{\gamma}^{-1}=(I_{H}- \gamma^{*}\gamma)^{-1/2}=I_{H}+\frac{1}{2}\gamma^{*}\gamma + \frac{3}{8}\gamma^{*}\gamma\gamma^{*}\gamma+ \cdots$ is in $\sigma(A)'$. So $\Delta_{\gamma}^{-1}-\eta ^{*}\gamma \Delta_{\gamma}^{-1}\in\sigma(A)'$. Next, note that $\eta_{1} ^{*}\eta_{2}\eta_{3}^{*}(\varphi(a)\otimes I_{H})=\eta_{1} ^{*}\eta_{2}\sigma(a)\eta_{3}^{*}=\eta_{1} ^{*}(\varphi(a)\otimes I_{H})\eta_{2}\eta_{3}^{*}=\sigma(a)\eta_{1} ^{*}\eta_{2}\eta_{3}^{*}$. So $\eta_{1} ^{*}\eta_{2}\eta_{3}^{*}\in (E^{\sigma})^*$. Then $(\gamma^{*} \Delta_{\gamma ^{*}}^{-1}-\eta ^{*}\Delta_{\gamma ^{*}}^{-1})(u\otimes I_{H})=(\gamma^{*}-\eta ^{*})(I_{E\overbar{\otimes}_{A}H}- \gamma\gamma^{*})^{-1/2}(u\otimes I_{H})=(\gamma^{*}-\eta ^{*})(I_{E\overbar{\otimes}_{A}H}+\frac{1}{2}\gamma\gamma^{*} + \frac{3}{8}\gamma\gamma^{*}\gamma\gamma^{*}+ \cdots)(u\otimes I_{H})=((\gamma^{*}-\eta ^{*})+\frac{1}{2}(\gamma^{*}-\eta ^{*})\gamma\gamma^{*}+ \frac{3}{8}(\gamma^{*}-\eta ^{*})\gamma\gamma^{*}\gamma\gamma^{*}+ \cdots)(u\otimes I_{H})$ is also in $(E^{\sigma})^*$.
\end{proof}
Note that
$T=
\left( 
\begin{matrix} 
\Delta_{\gamma}^{-1} & \gamma^{*} \Delta_{\gamma ^{*}}^{-1}(u\otimes I_{H}) \\
-\gamma \Delta_{\gamma}^{-1} & -\Delta_{\gamma ^{*}}^{-1}(u\otimes I_{H})
\end{matrix}
\right)$ not only acts on the matrix $(I_{H} \quad \eta ^{*})$. It acts on the equivalence class $[I_{H} \quad \eta ^{*}]$. That is, no matter what representative of $[I_{H} \quad \eta ^{*}]$ is multiplied by $T$, the matrix product always equals $[I_{H} \quad g(\eta ^{*})]$:
\begin{equation*}
\begin{split}
[C \quad C\eta ^{*}]&
\left( 
\begin{matrix} 
\Delta_{\gamma}^{-1} & \gamma^{*} \Delta_{\gamma ^{*}}^{-1}(u\otimes I_{H}) \\
-\gamma \Delta_{\gamma}^{-1} & -\Delta_{\gamma ^{*}}^{-1}(u\otimes I_{H})
\end{matrix}
\right)\\
&=[C(\Delta_{\gamma}^{-1}-\eta ^{*}\gamma \Delta_{\gamma}^{-1}) \qquad C(\gamma^{*} \Delta_{\gamma ^{*}}^{-1}-\eta ^{*}\Delta_{\gamma ^{*}}^{-1}(u\otimes I_{H}))] \\
&=[\Delta_{\gamma}^{-1}-\eta ^{*}\gamma \Delta_{\gamma}^{-1} \qquad \gamma^{*} \Delta_{\gamma ^{*}}^{-1}-\eta ^{*}\Delta_{\gamma ^{*}}^{-1}(u\otimes I_{H})] \\
&= [I_{H} \qquad g(\eta ^{*})]
\end{split}
\end{equation*}
As we stated above, in general, any $\g\in Aut(\mathbb{D}({E^{\sigma}}^*))$ (not just one implementing an automorphism  $\alpha \in Aut(H^{\infty}(E))$ ) is of the form $g=\omega\circ  g_{\gamma}$. The linear isometries $\omega$ on $Aut(\mathbb{D}({E^{\sigma}}^*))$ are given by $\omega(\eta ^{*})=u\eta ^{*}v^{*}$, where $u$ and $v$ are unitaries in $B(H)$ and $B(E\overbar{\otimes}_{A}H)$ respectively, satisfying additional conditions ensuring that $u\eta ^{*}v^{*}\in({E^{\sigma}})^*$. For example, in the case when $(E,A)$ be a $W^{*}$ graph correspondence, we have the following result:
\begin{theorem}\label{isometries}
Let $(E,A)$ be a $W^{*}$ correspondence derived from a directed graph $G=(G^{0}, G^{1},r,s)$ (without sources). Let $\omega$ be a linear isometry of $(E^{\sigma})^*$ given by $\omega(\eta^{*})=u\eta^{*}v^{*}$, where $u$ and $v$ are unitaries on $H$ and $E \overbar{\otimes}H$ respectively. Then $u$ is a diagonal block matrix $u=\bigoplus \limits^{|G^{0}|}_{i=1}u_{i}$, where $u_{i}$ is a unitary on $H_{i}$, and $v^{*}$ is a block matrix $v^{*}= (v_{ij}^{*})$ satisfying the following:
\begin{enumerate}
\item $v_{ij}^{*}=0$ if $r(e_{i})\neq r(e_{j})$
\item For each $i \in \{1,2,\cdots|G^{1}| \}$, $\sum \limits^{|G^{1}|}_{j=1}v_{ij}^{*}v_{ij}=I_{H_{s(e_{i})}}$
\item For each $i,k (i \neq k)\in \{1,2,\cdots|G^{1}| \}$, $\sum \limits^{|G^{1}|}_{j=1}v_{ij}^{*}v_{kj}=0$
\end{enumerate}
\end{theorem}
\begin{proof}
Suppose $u$ has a nonzero off diagonal block $u_{ij}$. So $i\neq j$. Since we are assuming $G$ does not have any sources, the vertex $v_{j}$ is the range of some edge $e_{k}$. That is, $\eta^{*}$ has a nonzero block $\eta(e_{k}^{-1})^{*}$ on row $j$ and column $k$. Then $u\eta^{*}$ has the nonzero block $u_{ij}\eta(e_{k}^{-1})^{*}$ on row $i$ and column $k$, which is a contradiction, since by \cite[Theorem 4.1]{Ardila2019}, the only nonzero block on column $k$ of $\eta^{*}$ is on row $j\neq i$. Thus $u_{ij}=0$ for $i\neq j$. Now denote each diagonal block $u_{ii}$ by $u_{i}$. So $u=\bigoplus \limits^{|G^{0}|}_{i=1}u_{i}$. In the product $u\eta^{*}$, $u_{i}$ multiplies all blocks on row $i$ of $\eta^{*}$. Since all blocks on row $i$ of $\eta^{*}$ have range in $H_{i}$ and $u\eta^{*}\in (E^{\sigma})^*$, we must have $u_{i}:H_{i}\to H_{i}$. Since $u=\bigoplus \limits^{|G^{0}|}_{i=1}u_{i}$ is unitary, each $u_{i}$ is unitary.

Now let $v^{*}= (v_{ij}^{*})$. Note that $(v_{ij}^{*})$ is a $|G^{1}|$ by $|G^{1}|$ block matrix where $v_{ij}^{*}\in B(H_{s(e_{j})},H_{s(e_{j})})$. Let $r(e_{m})=v_{n}$. Then for $p\neq n$, the product of row $p$ of $\eta^{*}$ and column $m$ of $(v_{ij}^{*})$ equals $0$, by \cite[Theorem 4.1]{Ardila2019}. Since $\eta^{*}$ is arbitrary in $(E^{\sigma})^*$ and the entries in column $m$ of $(v_{ij}^{*})$ which multiply a nonzero entry in row $p$ of $\eta^{*}$, are the entries $v_{ij}^{*}$ such that $r(e_{i})=r(e_{p})\neq r(e_{j})$, we have that $v_{ij}^{*}=0$ if $r(e_{i})\neq r(e_{j})$. Furthermore, since $(v_{ij}^{*})(v_{ji})=I_{E\overbar{\otimes}_{A}H}$, we have that for each $i \in \{1,2,\cdots|G^{1}| \}$, $\sum \limits^{|G^{1}|}_{j=1}v_{ij}^{*}v_{ij}=I_{H_{s(e_{i})}}$, and since the non-diagonal entries of $(v_{ij}^{*})(v_{ji})=I_{E\overbar{\otimes}_{A}H}$ are $0$, we have that if $i\neq k$ then $\sum \limits^{|G^{1}|}_{j=1}v_{ij}^{*}v_{kj}=0$.
\end{proof}
Let $J=\{u\in B(H), v\in  B(E\overbar{\otimes}_{A}H) \enspace  | \enspace  u\eta^{*}v^{*} \in (E^{\sigma})^{*} \}$.

Let $\mathcal{M}=\{
\left( 
\begin{matrix} 
\Delta_{\gamma}^{-1}u^{*} & \gamma^{*} \Delta_{\gamma ^{*}}^{-1}v^{*} \\
-\gamma \Delta_{\gamma}^{-1}u^{*} & -\Delta_{\gamma ^{*}}^{-1}v^{*}
\end{matrix}
\right)\enspace | \enspace \gamma ^{*}\in \mathbb{D}((E^{\sigma})^{*}),u,v \in J \}$.

Let 
$\left( 
\begin{matrix} 
\Delta_{\gamma}^{-1}u_{1}^{*} & \gamma^{*} \Delta_{\gamma ^{*}}^{-1}v_{1}^{*} \\
-\gamma \Delta_{\gamma}^{-1}u_{1}^{*} & -\Delta_{\gamma ^{*}}^{-1}v_{1}^{*}
\end{matrix}
\right)\sim
\left( 
\begin{matrix} 
\Delta_{\gamma}^{-1}u_{2}^{*} & \gamma^{*} \Delta_{\gamma ^{*}}^{-1}v_{2}^{*} \\
-\gamma \Delta_{\gamma}^{-1}u_{2}^{*} & -\Delta_{\gamma ^{*}}^{-1}v_{2}^{*}
\end{matrix}
\right)$ if $u_{1}\eta^{*}v_{1}^{*}=u_{2}\eta^{*}v_{2}^{*}$ for all $\eta^{*}\in \mathbb{D}({E^{\sigma}}^*)$. That is, two matrices in $\mathcal{M}$ are equivalent if the unitaries in each matrix determine the same isometry $\omega$ of $\mathbb{D}({E^{\sigma}}^*)$. 
Let $\mathcal{M}^{op}$ denote the opposite group of $\mathcal{M} $. That is, $\mathcal{M}^{op}$ has the same elements as $\mathcal{M}$ and the group operation $*$ is defined by reversing the matrix multiplication order. So for any $T,S\in \mathcal{M}^{op}$, we have $T*S=ST$. As groups, $\mathcal{M}^{op} \cong \mathcal{M}$, with isomorphism given by $\pi(T)=T^{*}$. Furthermore, $\mathcal{M}^{op}/\sim$ is a group with the operation $[T]*[S]=[ST]$ and identity element
 $[\left( 
\begin{matrix} 
I_{H} &  \\
0 & I_{E\overbar{\otimes}_{A}H}
\end{matrix}
\right)]$.

\begin{theorem}
Let $g=\omega\circ  g_{\gamma}\in Aut(\mathbb{D}({E^{\sigma}}^*))$ where $\omega$ is determined by unitaries $u$ and $v$. The map $\Psi: Aut(\mathbb{D}({E^{\sigma}}^*))\to \mathcal{M}^{op}/\sim$ defined by
$\Psi(g)=[
\left( 
\begin{matrix} 
\Delta_{\gamma}^{-1}u^{*} & \gamma^{*} \Delta_{\gamma ^{*}}^{-1}v^{*} \\
-\gamma \Delta_{\gamma}^{-1}u^{*} & -\Delta_{\gamma ^{*}}^{-1}v^{*}
\end{matrix}
\right)]$ is a group isomorphism. 
\end{theorem}
\begin{proof}
First we show that the matrices representating $g_{\gamma}$ and $\omega$ are
$\left( 
\begin{matrix} 
\Delta_{\gamma}^{-1} & \gamma^{*} \Delta_{\gamma ^{*}}^{-1} \\
-\gamma \Delta_{\gamma}^{-1} & -\Delta_{\gamma ^{*}}^{-1}
\end{matrix}
\right)$ and
$\left( 
\begin{matrix} 
u^{*} & 0\\
0 & v^{*}
\end{matrix}
\right)$ respectively:
\begin{equation*}
\begin{split}
[I_{H} \quad \eta ^{*}]
\left( 
\begin{matrix} 
\Delta_{\gamma}^{-1} & \gamma^{*} \Delta_{\gamma ^{*}}^{-1} \\
-\gamma \Delta_{\gamma}^{-1} & -\Delta_{\gamma ^{*}}^{-1}
\end{matrix}
\right)&=[\Delta_{\gamma}^{-1}-\eta ^{*}\gamma \Delta_{\gamma}^{-1} \qquad \gamma^{*} \Delta_{\gamma ^{*}}^{-1}-\eta ^{*}\Delta_{\gamma ^{*}}^{-1}] \\
&= [I_{H} \qquad (\Delta_{\gamma}^{-1}-\eta ^{*}\gamma \Delta_{\gamma}^{-1})^{-1} (\gamma^{*} \Delta_{\gamma ^{*}}^{-1}-\eta ^{*}\Delta_{\gamma ^{*}}^{-1})]\\
&= [I_{H} \qquad ((I_{H}-\eta ^{*}\gamma) \Delta_{\gamma}^{-1})^{-1} (\gamma^{*}-\eta ^{*})\Delta_{\gamma ^{*}}^{-1}]\\
&= [I_{H} \qquad \Delta_{\gamma}(I_{H}-\eta ^{*}\gamma)^{-1} (\gamma^{*}-\eta ^{*})\Delta_{\gamma ^{*}}^{-1}]\\
&= [I_{H} \qquad g_{\gamma}(\eta ^{*})]
\end{split}
\end{equation*} and $[I_{H} \quad \eta ^{*}]
\left( 
\begin{matrix} 
u^{*} & 0\\
0 & v^{*}
\end{matrix}
\right)
=[u^{*} \quad \eta ^{*}v^{*}]=[I_{H} \quad u\eta ^{*}v^{*}]= [I_{H} \quad \omega(\eta ^{*})]$.

\noindent
Then $g=\omega\circ g_{\gamma} $ is represented by $
\left( 
\begin{matrix} 
\Delta_{\gamma}^{-1} & \gamma^{*} \Delta_{\gamma ^{*}}^{-1} \\
-\gamma \Delta_{\gamma}^{-1} & -\Delta_{\gamma ^{*}}^{-1}
\end{matrix}
\right)
\left( 
\begin{matrix} 
u^{*} & 0\\
0 & v^{*}
\end{matrix}
\right)= 
\left( 
\begin{matrix} 
\Delta_{\gamma}^{-1}u^{*} & \gamma^{*} \Delta_{\gamma ^{*}}^{-1}v^{*} \\
-\gamma \Delta_{\gamma}^{-1}u^{*} & -\Delta_{\gamma ^{*}}^{-1}v^{*}
\end{matrix}
\right)$:
\begin{equation*}
\begin{split}
[I_{H} \quad \eta ^{*}]
\left( 
\begin{matrix} 
\Delta_{\gamma}^{-1}u^{*} & \gamma^{*} \Delta_{\gamma ^{*}}^{-1}v^{*} \\
-\gamma \Delta_{\gamma}^{-1}u^{*} & -\Delta_{\gamma ^{*}}^{-1}v^{*}
\end{matrix}
\right)
&=[(\Delta_{\gamma}^{-1}-\eta ^{*}\gamma \Delta_{\gamma}^{-1})u^{*} \qquad (\gamma^{*} \Delta_{\gamma ^{*}}^{-1}-\eta ^{*}\Delta_{\gamma ^{*}}^{-1})v^{*}] \\
&= [I_{H} \qquad ((\Delta_{\gamma}^{-1}-\eta ^{*}\gamma \Delta_{\gamma}^{-1})u^{*})^{-1} (\gamma^{*} \Delta_{\gamma ^{*}}^{-1}-\eta ^{*}\Delta_{\gamma ^{*}}^{-1})v^{*}]\\
&= [I_{H} \qquad u((I_{H}-\eta ^{*}\gamma) \Delta_{\gamma}^{-1})^{-1} (\gamma^{*}-\eta ^{*})\Delta_{\gamma ^{*}}^{-1}v^{*}]\\
&= [I_{H} \qquad u\Delta_{\gamma}(I_{H}-\eta ^{*}\gamma)^{-1} (\gamma^{*}-\eta ^{*})\Delta_{\gamma ^{*}}^{-1}v^{*}]\\
&= [I_{H} \qquad ug_{\gamma}(\eta ^{*})v^{*}]=[I_{H} \qquad \omega(g_{\gamma}(\eta ^{*}))]\\
&= [I_{H} \qquad g(\eta ^{*})]
\end{split}
\end{equation*}
$\Psi$ is clearly onto. 
If $g$ and $f$ in $Aut(\mathbb{D}({E^{\sigma}}^*))$ are represented by matrices  $T_{g}$ and $T_{f}$ respectively, 
then we have $[I_{H} \quad \eta ^{*}]T_{f}T_{g}=[I_{H} \quad f(\eta ^{*})]T_{g}=[I_{H} \quad g(f(\eta ^{*}))]$. So $\Psi(g\circ f)=[T_{g\circ f}]=[T_{f}T_{g}]=[T_{g}]*[T_{f}]=\Psi(g)* \Psi(f)$. 
 Since by \cref{uniquedecomposition}, each $g\in Aut(\mathbb{D}({E^{\sigma}}^*))$ has a unique decomposition $g=\omega \circ g_{g^{-1}(o)^{*}}$, we have that $\Psi: Aut(\mathbb{D}({E^{\sigma}}^*))\to \mathcal{M}^{op}/\sim$ is injective, thus a group isomorphism. Note that since $\Psi(g^{-1})=\Psi(g_{\gamma}^{-1})*\Psi(\omega^{-1})=[T_{\omega}^{-1}T_{g_{\gamma}}]$, the inverse of each $[
\left( 
\begin{matrix} 
\Delta_{\gamma}^{-1}u^{*} & \gamma^{*} \Delta_{\gamma ^{*}}^{-1}v^{*} \\
-\gamma \Delta_{\gamma}^{-1}u^{*} & -\Delta_{\gamma ^{*}}^{-1}v^{*}
\end{matrix}
\right)]$ in $\mathcal{M}^{op}/\sim$ is
$[
\left( 
\begin{matrix} 
u\Delta_{\gamma}^{-1} & u\gamma^{*} \Delta_{\gamma ^{*}}^{-1}\\
-v\gamma \Delta_{\gamma}^{-1} & -v\Delta_{\gamma ^{*}}^{-1}
\end{matrix}
\right)]$.
\end{proof}


Let $K$ be a Hilbert space. A linear invertible operator $S$ is said to be $\kappa$-pseudo-unitary if there exists a linear invertible, Hermitian operator $\kappa :K\to K$ such that $S$ satisfies $S^{*}=\kappa S^{-1}\kappa^{-1}$. As in \cite[2]{Mostafazadeh2004}, we denote the group of all $\kappa$-pseudo-unitary operators by $\mathcal{U}_\kappa(K)$.

\begin{theorem}
$Aut(\mathbb{D}({E^{\sigma}}^*))$ and $Aut(H^{\infty}(E))$ are subgroups of the pseudo-unitary group $\mathcal{U}_\kappa(H\bigoplus E \overbar{\otimes}H)$ for $\kappa=\left( 
\begin{matrix} 
I_{H} & 0 \\
0 & -I_{E\overbar{\otimes}_{A}H}
\end{matrix}
\right)$.
\end{theorem}

\begin{proof}
If $\kappa=\left( 
\begin{matrix} 
I_{H} & 0 \\
0 & -I_{E\overbar{\otimes}_{A}H}
\end{matrix}
\right)$, then the condition $S^{*}=\kappa S^{-1}\kappa^{-1}$ in the definition of $\kappa$-pseudo-unitary operator is equivalent to $S\kappa S^{*}=\kappa$. We show that the matrices representing elements of $Aut(\mathbb{D}({E^{\sigma}}^*))$ and $Aut(H^{\infty}(E))$ satisfy this condition.

If $\gamma ^{*} \in \mathbb{D}({E^{\sigma}}^*))$, then the series $\Sigma_{n=0}^{\infty}(\gamma^{*}\gamma)^{n}$ converges in norm to the operator $(I_{H}-\gamma^{*}\gamma)^{-1}$. So $\Delta_{\gamma}^{-2}=(I_{H}-\gamma^{*}\gamma)^{-1}=\Sigma_{n=0}^{\infty}(\gamma^{*}\gamma)^{n}$. Likewise, $\Delta_{\gamma^{*}}^{-2}=(I_{E\overbar{\otimes}_{A}H}-\gamma\gamma^{*})^{-1}=\Sigma_{n=0}^{\infty}(\gamma\gamma^{*})^{n}$. Note that $\gamma^{*}\Delta_{\gamma ^{*}}^{-2}\gamma=\gamma^{*}\Sigma_{n=0}^{\infty}(\gamma\gamma^{*})^{n}\gamma=\Sigma_{n=1}^{\infty}(\gamma^{*}\gamma)^{n}$. Thus
\begin{equation}\label{one}
 \Delta_{\gamma}^{-2}-\gamma^{*}\Delta_{\gamma ^{*}}^{-2}\gamma=\Sigma_{n=0}^{\infty}(\gamma^{*}\gamma)^{n}-\Sigma_{n=1}^{\infty}(\gamma^{*}\gamma)^{n}=I_{H}
\end{equation}
Note that $\gamma \Delta_{\gamma}^{-2}\gamma^{*}=\gamma\Sigma_{n=0}^{\infty}(\gamma^{*}\gamma)^{n}\gamma^{*}=\Sigma_{n=1}^{\infty}(\gamma\gamma^{*})^{n}$. Thus
\begin{equation}\label{two}
\gamma\Delta_{\gamma }^{-2}\gamma^{*}-\Delta_{\gamma^{*}}^{-2}=\Sigma_{n=1}^{\infty}(\gamma\gamma^{*})^{n}-\Sigma_{n=0}^{\infty}(\gamma\gamma^{*})^{n}=-I_{E\overbar{\otimes}_{A}H}
\end{equation}
Note that $\Delta_{\gamma}^{-2}\gamma^{*}=\Sigma_{n=0}^{\infty}(\gamma^{*}\gamma)^{n}\gamma^{*}=\gamma^{*}\Sigma_{n=0}^{\infty}(\gamma\gamma^{*})^{n}$.Thus
\begin{equation}\label{three}
-\Delta_{\gamma}^{-2}\gamma^{*} +\gamma^{*}\Delta_{\gamma^{*}}^{-2}=-\gamma^{*}\Sigma_{n=0}^{\infty}(\gamma\gamma^{*})^{n}+\gamma^{*}\Sigma_{n=0}^{\infty}(\gamma\gamma^{*})^{n}=0
\end{equation}
Note that $\gamma\Delta_{\gamma}^{-2}=\gamma\Sigma_{n=0}^{\infty}(\gamma^{*}\gamma)^{n}=\Sigma_{n=0}^{\infty}(\gamma\gamma^{*})^{n}\gamma$. Thus
\begin{equation}\label{four}
-\gamma \Delta_{\gamma}^{-2}+\Delta_{\gamma ^{*}}^{-2}\gamma=-\Sigma_{n=0}^{\infty}(\gamma\gamma^{*})^{n}\gamma+\Sigma_{n=0}^{\infty}(\gamma\gamma^{*})^{n}\gamma=0
\end{equation}

Then we have
\begin{equation*}
\begin{split}
\left( 
\begin{matrix} 
\Delta_{\gamma}^{-1}u^{*} & \gamma^{*} \Delta_{\gamma ^{*}}^{-1}v^{*} \\
-\gamma \Delta_{\gamma}^{-1}u^{*} & -\Delta_{\gamma ^{*}}^{-1}v^{*}
\end{matrix}
\right)&
\left( 
\begin{matrix} 
I_{H} & 0 \\
0 & -I_{E\overbar{\otimes}_{A}H}
\end{matrix}
\right)
\left( 
\begin{matrix} 
(\Delta_{\gamma}^{-1}u^{*})^{*} & (-\gamma \Delta_{\gamma}^{-1}u^{*})^{*} \\
(\gamma^{*} \Delta_{\gamma ^{*}}^{-1}v^{*})^{*} & (-\Delta_{\gamma ^{*}}^{-1}v^{*})^{*}
\end{matrix}
\right)\\
&=
\left( 
\begin{matrix} 
\Delta_{\gamma}^{-1}u^{*} & -\gamma^{*} \Delta_{\gamma ^{*}}^{-1}v^{*} \\
-\gamma \Delta_{\gamma}^{-1}u^{*} & \Delta_{\gamma ^{*}}^{-1}v^{*}
\end{matrix}
\right)
\left( 
\begin{matrix} 
u\Delta_{\gamma}^{-1} & -u\Delta_{\gamma}^{-1}\gamma^{*}  \\
v\Delta_{\gamma ^{*}}^{-1}\gamma & -v\Delta_{\gamma ^{*}}^{-1}
\end{matrix}
\right)\\
&=
\left( 
\begin{matrix} 
\Delta_{\gamma}^{-2}-\gamma^{*}\Delta_{\gamma ^{*}}^{-2}\gamma & -\Delta_{\gamma}^{-2}\gamma^{*} +\gamma^{*}\Delta_{\gamma^{*}}^{-2} \\
-\gamma \Delta_{\gamma}^{-2}+\Delta_{\gamma ^{*}}^{-2}\gamma & \gamma\Delta_{\gamma }^{-2}\gamma^{*}-\Delta_{\gamma^{*}}^{-2}
\end{matrix}
\right)\\
&=
\left( 
\begin{matrix} 
I_{H} & 0 \\
0 & -I_{E\overbar{\otimes}_{A}H}
\end{matrix}
\right)
\end{split}
\end{equation*}
The last equality follows from equations (\ref{one}), (\ref{two}), (\ref{three}) and (\ref{four}).
\end{proof}

Since the elements of $Aut(H^{\infty}(E))$ are determined by the elements in $Aut$ $(\mathbb{D}(E^{\sigma})^*)$ preserving $\mathfrak{Z}({D}(E^{\sigma})^*)$, we can  think of $Aut(H^{\infty}(E))$ as a subgroup of $Aut(\mathbb{D}(E^{\sigma})^*)$. 
Is $Aut(H^{\infty}(E))$ a normal subgroup of $Aut(\mathbb{D}(E^{\sigma})^*)$? We answer this question for the case when $(E,A)$ is a  $W^{*}$-graph correspondence.
\begin{lemma}\label{isometriesnotnormal}
Let $N=\{\omega\in  Aut(\mathbb{D}({E^{\sigma}}^*))$ $|$ $\omega$ is an isometry $ \}$ 
 Then $N\ntrianglelefteq Aut(\mathbb{D}(E^{\sigma})^*)$ 
\end{lemma}
\begin{proof}
Let $\gamma^{*}\in \mathbb{D}(E^{\sigma})^*$ with $\gamma^{*}\neq 0$.Then $g_{\gamma} \circ -id \circ g_{\gamma} (0)=g_{\gamma}(-\gamma^{*})\neq 0$. So $g_{\gamma}$ does not normalize $N$.
\end{proof}
\begin{theorem}
Let $(E,A)$ be a $W^{*}$-graph correspondence derived from a graph $G$ (and assume that the multiplicity of the representation of $\delta_{v}$ is greater than 1 for at least one vertex). Then $Aut(H^{\infty}(E))\ntrianglelefteq Aut(\mathbb{D}(E^{\sigma})^*)$.
\end{theorem}
\begin{proof}
If $\mathfrak{Z}((E^{\sigma})^{*})=\{0\}$ then $G$ does not have any loops and $Aut(H^{\infty}(E))=N=\{\omega\in  Aut(\mathbb{D}({E^{\sigma}}^*))$ $|$ $\omega$ is an isometry $ \}  \ntrianglelefteq Aut(\mathbb{D}(E^{\sigma})^*)$ 
by \cref{isometriesnotnormal}. So we may assume that $\mathbb{D} \mathfrak{Z}((E^{\sigma})^{*})\neq \emptyset$. By \cite[Corollary 4.2]{Ardila2019} and \cref{isometries}, since the multiplicity of the representation of $\delta_{v}$ is greater than 1 for at least one vertex in $G$, we have that $\mathfrak{Z}((E^{\sigma})^{*}) \neq (E^{\sigma})^{*}$ 
and there is a linear isometry $\omega \in Aut((E^{\sigma})^*)$  not preserving $\mathfrak{Z}((E^{\sigma})^{*})$ 
Let $\gamma^{*} \in \mathbb{D} \mathfrak{Z}((E^{\sigma})^{*})$ such that $\omega(\gamma^{*}) \notin \mathbb{D} \mathfrak{Z}((E^{\sigma})^{*})$. Then $\omega \circ g_{\gamma} \circ \omega^{-1}(0)=\omega(\gamma^{*})  \notin \mathbb{D} \mathfrak{Z}((E^{\sigma})^{*})$. So $\omega \circ g_{\gamma} \circ \omega^{-1}$ does not preserve $\mathbb{D} \mathfrak{Z}((E^{\sigma})^{*})$. Thus $\omega$ does not normalize $Aut(H^{\infty}(E))$. So $Aut(H^{\infty}(E))\ntrianglelefteq Aut(\mathbb{D}(E^{\sigma})^*)$.
\end{proof}

\section{A Morita Equivalence Application}
In this section, we present an application of the automorphism group $Aut(\mathbb{D}({F^{\sigma}}^*))$ in the study of Morita equivalence of $W^{*}$-correspondences. One of the important features of Morita equivalence of rings, $C^{*}$-algebras, $W^{*}$-algebras, operator algebras, etc,  is that Morita equivalent objects have the same representation theory. More precisely, having Morita equivalent objects implies there is an equivalence between the categories of the relevant representation of those objects. Our goal in this section is to show that if two $W^{*}$-correspondences are weakly Morita equivalent then there is an equivalence between  certain categories of completely contractive covariant representations of both correspondences.

Let $(E, A)$ and $(F, B)$ be two (weakly) Morita equivalent $W^{*}$-correspondences. So there is a $W^{*}$-equivalence bimodule $X$ for which there is an $A$-$B$ \thinspace $W^{*}$-correspondence isomorphism $W$ from $E\overbar{\otimes}_{A}X$ onto $X \overbar{\otimes}_{B}F$ \cite[Definition 7]{Muhly2011b}. Let $\sigma:B\to B(H)$ be a normal representation of $B$, and let $\sigma ^{X}:A\to B(X \overbar{\otimes}_{\sigma}H)$ be the normal representation of $A$  induced by $X$ (Rieffel's induced representation). Let $\eta ^{*} \in \mathbb{D}({F^{\sigma}}^*)$. In \cite[3]{Muhly2011b}, Muhly and Solel showed that the map $\eta ^{*}\to {\eta ^{*}}^{X}$, where ${\eta ^{*}}^{X}=(I_{X}\otimes\eta ^{*})(W\otimes I_{H})$, is an isometric surjection from $\overline{\mathbb{D}({F^{\sigma}}^*)}$ onto $\overline{\mathbb{D}({E^{\sigma ^{X}*}})}$. Thus, this map is also an isometric surjection from $\mathbb{D}({F^{\sigma}}^*)$ onto $\mathbb{D}({E^{\sigma ^{X}*}})$.


Let $\sigma$-Covrep$F$ denote the category whose objects are the intertwiners in $\mathbb{D}({F^{\sigma}}^*)$. By \cite[Theorem 2.9 and Corollary 2.14]{Muhly2004a}, we can think of these objects as completely contractive covariant representations of $F$ (associated to $\sigma$) implementing normal completely contractive representations of $H^{\infty}(F)$. We point out that this set is a subset of the set of absolutely continuous completely contractive representations of $F$ (\cite[Definition 3.1]{Muhly2011a}), which are the representations implementing all extensions of completely contractive representations of $\mathcal{T}_{+}(E)$ to normal completely contractive representations of $H^{\infty}(F)$ (\cite[Theorem 4.11]{Muhly2011a}).


 Let the morphisms of $\sigma$-Covrep$F$ be given by:
\begin{equation*}
Hom(\eta_{1} ^{*},\eta_{2} ^{*})=\{g \in Aut(\mathbb{D}({F^{\sigma}}^*))\enspace | \enspace g(\eta_{1} ^{*})=\eta_{2} ^{*}\}
\end{equation*}

First we check that $\sigma$-Covrep$F$ is indeed a category. The composition of morphisms is given by the usual composition of  maps in $Aut(\mathbb{D}({F^{\sigma}}^*))$. If $\eta_{1} ^{*},\eta_{2} ^{*}, \eta_{3} ^{*}\in \mathbb{D}({F^{\sigma}}^*)$, $f \in$ Hom$(\eta_{1} ^{*},\eta_{2} ^{*})$ and $g \in$ Hom$(\eta_{2} ^{*},\eta_{3} ^{*})$, then $g\circ f \in$ Hom$(\eta_{1} ^{*},\eta_{3} ^{*})$, since $g\circ f (\eta_{1} ^{*})=\eta_{3}$. Since $Aut(\mathbb{D}({F^{\sigma}}^*))$ is a group, this composition of morphisms is associative. Clearly, the identity map in $Aut(\mathbb{D}({F^{\sigma}}^*))$ serves as an identity morphism Id$_{\eta ^{*}}\in$ Hom$(\eta ^{*},\eta ^{*})$ for each  $\eta ^{*}\in \mathbb{D}({F^{\sigma}}^*)$. 
Likewise, the category $\sigma ^{X}$-Covrep$E$ has the intertwiners in $\mathbb{D}({E^{\sigma ^{X}*}})$ as objects, and the morphisms are given by  Hom$({\eta_{1} ^{*}}^{X},{\eta_{2} ^{*}}^{X})=\{g \in Aut(\mathbb{D}({E^{\sigma ^{X}*}})) \enspace | \enspace g({\eta_{1} ^{*}}^{X})={\eta_{2} ^{*}}^{X}\}$.
\begin{theorem}\label{category}
If $(E, A)$ and $(F, B)$ are two (weakly) Morita equivalent $W^{*}$- correspondences and $\sigma:B\to B(H)$ is a normal representation of $B$, then $\sigma$-Covrep$F$ and $\sigma ^{X}$-Covrep$E$ are equivalent categories.
\end{theorem}
\begin{proof}
Let $\mathscr{F}$ denote the functor from $\sigma$-Covrep$F$ to $\sigma ^{X}$-Covrep$E$, given by $\mathscr{F}(\eta ^{*} )= {\eta ^{*}}^{X}$. 
For each $g\in Aut(\mathbb{D}({F^{\sigma}}^*))$,  $\mathscr{F}(g)$ is defined by $\mathscr{F}(g)({\eta ^{*}}^{X})= (g{(\eta ^{*}}))^{X}$. That is , $\mathscr{F}(g)((I_{X}\otimes\eta ^{*})(W\otimes I_{H}))=(I_{X}\otimes g(\eta ^{*}))(W\otimes I_{H}))$. So if $g\in$ Hom$(\eta_{1} ^{*},\eta_{2} ^{*})$, then $\mathscr{F}(g)\in$ Hom$({\eta_{1} ^{*}}^{X},{\eta_{2} ^{*}}^{X})=$ Hom$(\mathscr{F}(\eta_{1} ^{*}),\mathscr{F}(\eta_{2} ^{*}))$. Clearly, $\mathscr{F}($Id$_{\eta ^{*}})=$Id$_{\mathscr{F}(\eta ^{*})}$. If $g,h\in Aut(\mathbb{D}({F^{\sigma}}^*))$ then
\begin{align*}
\mathscr{F}(g\circ h)((I_{X}\otimes\eta ^{*})(W\otimes I_{H}))&=(I_{X}\otimes (g\circ h)(\eta ^{*}))(W\otimes I_{H})\\
&=(I_{X}\otimes (g(h(\eta ^{*})))(W\otimes I_{H})\\
&=\mathscr{F}(g)((I_{X}\otimes (h(\eta ^{*}))(W\otimes I_{H}))\\
&=\mathscr{F}(g)(\mathscr{F}(h)((I_{X}\otimes \eta ^{*})(W\otimes I_{H})))\\
&=(\mathscr{F}(g)\circ \mathscr{F}(h))((I_{X}\otimes \eta ^{*})(W\otimes I_{H}))
\end{align*}
So $\mathscr{F}(g\circ h)=\mathscr{F}(g)\circ \mathscr{F}(h)$. Thus $\mathscr{F}$ is an isometric covariant functor from $\sigma$-Covrep$F$ to $\sigma ^{X}$-Covrep$E$.

Since $W:E\otimes _{A}X \to X \otimes _{B}F $,is a $W^{*}$-correspondence isomorphism, there is an isomorphism $W':{_{B}F}_{B} \to {_{B}\widetilde{X}}\otimes _{A}E\otimes _{A}X_{B} $. Let $\mathscr{G}$ be the functor from $\sigma ^{X}$-Covrep$E$ to $\sigma$-Covrep$F$ given by $\mathscr{G}({\eta ^{*}}^{X})=(I_{\widetilde{X}}\otimes {\eta ^{*}}^{X})(W'\otimes I_{H})$. That is, $\mathscr{G}((I_{X}\otimes \eta ^{*})(W\otimes I_{H}))=(I_{\widetilde{X}}\otimes (I_{X}\otimes \eta ^{*})(W\otimes I_{H}))(W'\otimes I_{H})$.

\noindent
Since $||(I_{\widetilde{X}}\otimes {\eta ^{*}}^{X})(W'\otimes I_{H})||=||{\eta ^{*}}^{X}||$, $(I_{\widetilde{X}}\otimes {\eta ^{*}}^{X})(W'\otimes I_{H})$ lies in $\mathbb{D}({E^{\sigma *}})$.

\noindent
For any $g \in Aut(\mathbb{D}({E^{\sigma ^{X}*}}))$, $\mathscr{G}(g)\in Aut(\mathbb{D}({F^{\sigma}}^*))$ is defined by:
\begin{align*}
\mathscr{G}(g)((I_{\widetilde{X}}\otimes {\eta ^{*}}^{X})(W'\otimes I_{H}))=(I_{\widetilde{X}}\otimes g({\eta ^{*}}^{X}))(W'\otimes I_{H})
\end{align*}
Then $\mathscr{G}($Id$_{{\eta ^{*}}^{X}})=$Id$_{\mathscr{G}({\eta ^{*}}^{X})}$ and 
\begin{align*}
\mathscr{G}(g\circ h)((I_{\widetilde{X}}\otimes {\eta ^{*}}^{X})(W'\otimes I_{H}))&=((I_{\widetilde{X}}\otimes (g\circ h)({\eta ^{*}}^{X}))(W'\otimes I_{H}))\\
&=((I_{\widetilde{X}}\otimes g(h({\eta ^{*}}^{X})))(W'\otimes I_{H}))\\
&= \mathscr{G}(g)((I_{\widetilde{X}}\otimes h({\eta ^{*}}^{X}))(W'\otimes I_{H}))\\
&=\mathscr{G}(g)(\mathscr{G}(h)((I_{\widetilde{X}}\otimes {\eta ^{*}}^{X})(W'\otimes I_{H})))\\
&=(\mathscr{G}(g)\circ \mathscr{G}(h))((I_{\widetilde{X}}\otimes {\eta ^{*}}^{X})(W'\otimes I_{H}))
\end{align*}
Thus $\mathscr{G}(g\circ h)=\mathscr{G}(g)\circ \mathscr{G}(h)$. Also, if ${\eta_{1} ^{*}}^{X}, {\eta_{2} ^{*}}^{X}\in \sigma ^{X}$-Covrep$E$ and $g\in$  Hom$_{\sigma ^{X}-CovrepE}$ $({\eta_{1} ^{*}}^{X},{\eta_{2} ^{*}}^{X})$, then
\begin{align*}
\mathscr{G}(g)(\mathscr{G}({\eta_{1} ^{*}}^{X}))&=\mathscr{G}(g)((I_{\widetilde{X}}\otimes {\eta_{1} ^{*}}^{X})(W'\otimes I_{H}))=(I_{\widetilde{X}}\overbar {\otimes}g({\eta_{1} ^{*}}^{X}))(W'\otimes I_{H}))\\
&=(I_{\widetilde{X}}\otimes {\eta_{2} ^{*}}^{X})(W'\otimes I_{H})=\mathscr{G}({\eta_{2} ^{*}}^{X})
\end{align*}
So $\mathscr{G}(g)\in$ Hom$_{\sigma -CovrepF}$ $(\mathscr{G}({\eta_{1} ^{*}}^{X}), \mathscr{G}({\eta_{2} ^{*}}^{X})$). Thus $\mathscr{G}$ is an isometric covariant functor from $\sigma ^{X}$-Covrep$E$ to  $\sigma$-Covrep$F$.

Now we show that $\mathscr{F}$ and $\mathscr{G}$ are inverses of each other. That is, $\mathscr{F}$ and $\mathscr{G}$ implement and equivalence between the categories  $\sigma$-Covrep$F$ and $\sigma ^{X}$-Covrep$E$. The natural transformation $\epsilon:I_{\sigma-CovrepF}\to \mathscr{G}\circ \mathscr{F}$, where $I_{\sigma-CovrepF}$ denotes the identity functor on $\sigma$-Covrep$F$, is given by $\epsilon_{\eta ^{*}}(I_{\sigma-CovrepF}(\eta ^{*}))= (\mathscr{G}\circ \mathscr{F})(\eta ^{*})= (I_{\widetilde{X}}\otimes (I_{X}\otimes \eta ^{*})(W\otimes I_{H}))(W'\otimes I_{H}) $.

Let $\eta_{1} ^{*},\eta_{2} ^{*} \in \mathbb{D}({F^{\sigma}}^*)$ and $g\in$Hom$( \eta_{1} ^{*},\eta_{2} ^{*})$. Then we have
\begin{align*}
\epsilon_{\eta_{2} ^{*}}\circ I_{\sigma-CovrepF}(g)(I_{\sigma-CovrepF}(\eta_{1} ^{*}))&=\epsilon_{\eta_{2} ^{*}}(g(\eta_{1} ^{*}))=\epsilon_{\eta_{2} ^{*}}(\eta_{2} ^{*})\\
&=(I_{\widetilde{X}}\otimes (I_{X}\otimes \eta_{2} ^{*})(W\otimes I_{H}))(W'\otimes I_{H})\\
&=\mathscr{G}((I_{X}\otimes \eta_{2} ^{*})(W\otimes I_{H}))\\
&=\mathscr{G}((I_{X}\otimes g(\eta_{1} ^{*}))(W\otimes I_{H}))\\
&=\mathscr{G}(\mathscr{F}(g)((I_{X}\otimes \eta_{1} ^{*})(W\otimes I_{H})))\\
&=\mathscr{G}(\mathscr{F}(g))(\mathscr{G}((I_{X}\otimes \eta_{1} ^{*})(W\otimes I_{H}))\\
&=(\mathscr{G}\circ \mathscr{F})(g)((I_{\widetilde{X}}\otimes (I_{X}\otimes \eta_{1} ^{*})(W\otimes I_{H}))(W'\otimes I_{H}))\\
&=(\mathscr{G}\circ \mathscr{F})(g)\circ \epsilon_{\eta_{1} ^{*}}(I_{\sigma-CovrepF}(\eta_{1} ^{*}))
\end{align*}
Thus $\epsilon_{\eta_{2} ^{*}}\circ I_{\sigma-CovrepF}(g)=(\mathscr{G}\circ \mathscr{F})(g)\circ \epsilon_{\eta_{1} ^{*}}$. That is, the following diagram commutes.
\[
\begin{tikzcd}
I_{\sigma-CovrepF}(\eta_{1} ^{*}) \arrow{r}{g}
\arrow[swap]{d}{\epsilon_{\eta_{1} ^{*}}} & I_{\sigma-CovrepF}(\eta_{2} ^{*}) \arrow{d}{\epsilon_{\eta_{2} ^{*}}} \\
(\mathscr{G}\circ \mathscr{F})({\eta_{1} ^{*}}) \arrow{r}{(\mathscr{G}\circ \mathscr{F})(g)}& (\mathscr{G}\circ \mathscr{F})({\eta_{2} ^{*}})
\end{tikzcd}
\]
So indeed, $\epsilon$ is a natural transformation from $I_{\sigma-CovrepF}$ to $\mathscr{G}\circ \mathscr{F}$.

The natural transformation $\lambda:\mathscr{F}\circ \mathscr{G} \to I_{\sigma^{X}-CovrepE}$, where $I_{\sigma^{X}-CovrepE}$ denotes the identity functor on $\sigma^{X}-CovrepE$, is given by $\lambda_{{\eta ^{*}}^{X}}((\mathscr{F}\circ \mathscr{G})({\eta ^{*}}^{X}))={\eta ^{*}}^{X}$. That is,
\begin{align*}
\lambda_{{\eta ^{*}}^{X}}((I_{X}\otimes (I_{\widetilde{X}}\otimes {\eta ^{*}}^{X})(W'\otimes I_{H}))(W\otimes I_{H})
 )={\eta ^{*}}^{X}
\end{align*}
Let ${\eta_{1} ^{*}}^{X}, {\eta_{2} ^{*}}^{X}\in \mathbb{D}({E^{\sigma ^{X}*}})$ and $g\in$Hom$({\eta_{1} ^{*}}^{X}, {\eta_{2} ^{*}}^{X} )$. Then we have
\begin{align*}
\lambda_{{\eta_{2} ^{*}}^{X}}\circ &(\mathscr{F}\circ \mathscr{G})(g)((\mathscr{F}\circ \mathscr{G})({\eta_{1} ^{*}}^{X})\\
&=\lambda_{{\eta_{2} ^{*}}^{X}}\circ (\mathscr{F}\circ \mathscr{G})(g)((I_{X}\otimes (I_{\widetilde{X}}\otimes {\eta_{1} ^{*}}^{X})(W'\otimes I_{H}))(W\otimes I_{H})) \\
&=\lambda_{{\eta_{2} ^{*}}^{X}}  (\mathscr{F}(\mathscr{G}(g))(\mathscr{F}((I_{\widetilde{X}}\otimes {\eta_{1} ^{*}}^{X})(W'\otimes I_{H})))) \\
&=\lambda_{{\eta_{2} ^{*}}^{X}}  (\mathscr{F}(\mathscr{G}(g)((I_{\widetilde{X}}\otimes {\eta_{1} ^{*}}^{X})(W'\otimes I_{H}))))  \\
&=\lambda_{{\eta_{2} ^{*}}^{X}}  (\mathscr{F}((I_{\widetilde{X}}\otimes {\eta_{2} ^{*}}^{X})(W'\otimes I_{H}))) \\
&=\lambda_{{\eta_{2} ^{*}}^{X}} ((I_{X}\otimes (I_{\widetilde{X}}\otimes {\eta_{2} ^{*}}^{X})(W'\otimes I_{H}))(W\otimes I_{H}))  \\
&={\eta_{2} ^{*}}^{X}\\
&=g({\eta_{1} ^{*}}^{X})\\
&=I_{\sigma^{X}-CovrepE}(g)\circ \lambda_{{\eta_{1} ^{*}}^{X}}
((I_{X}\otimes (I_{\widetilde{X}}\otimes {\eta_{1} ^{*}}^{X})(W'\otimes I_{H}))(W\otimes I_{H})) \\
&=I_{\sigma^{X}-CovrepE}(g)\circ \lambda_{{\eta_{1} ^{*}}^{X}}
(\mathscr{F}((I_{\widetilde{X}}\otimes {\eta_{1} ^{*}}^{X})(W'\otimes I_{H}))) \\
&=I_{\sigma^{X}-CovrepE}(g)\circ \lambda_{{\eta_{1} ^{*}}^{X}}((\mathscr{F}\circ \mathscr{G})((I_{\widetilde{X}}\otimes {\eta_{1} ^{*}}^{X})(W'\otimes I_{H})))\\
&=I_{\sigma^{X}-CovrepE}(g)\circ \lambda_{{\eta_{1} ^{*}}^{X}}((\mathscr{F}\circ \mathscr{G})({\eta_{1} ^{*}}^{X}))
\end{align*}
Thus $\lambda_{{\eta_{2} ^{*}}^{X}}\circ (\mathscr{F}\circ \mathscr{G})(g)=I_{\sigma^{X}-CovrepE}(g)\circ \lambda_{{\eta_{1} ^{*}}^{X}}$. That is, the following diagram commutes.
\[
\begin{tikzcd}
(\mathscr{F}\circ \mathscr{G})({\eta_{1} ^{*}}^{X}) \arrow{r}{g}
\arrow[swap]{d}{\lambda_{{\eta_{1} ^{*}}^{X}}} & (\mathscr{F}\circ \mathscr{G})({\eta_{2} ^{*}}^{X}) \arrow{d}{\lambda_{{\eta_{2} ^{*}}^{X}}} \\
I_{\sigma^{X}-CovrepE}(\eta_{1} ^{*}) \arrow{r}{(\mathscr{G}\circ \mathscr{F})(g)}& I_{\sigma^{X}-CovrepE}(\eta_{2} ^{*})
\end{tikzcd}
\]
So indeed, $\lambda$ is a natural transformation from $\mathscr{F}\circ \mathscr{G}$ to $I_{\sigma^{X}-CovrepE}$. 

\noindent The categories $\sigma$-Covrep$F$ and $\sigma ^{X}$-Covrep$E$ are equivalent.
\end{proof}

\bibliographystyle{amsalpha}
\bibliography{ReneMaster}

\def\cprime{$'$}
\providecommand{\bysame}{\leavevmode\hbox to3em{\hrulefill}\thinspace}
\providecommand{\MR}{\relax\ifhmode\unskip\space\fi MR }
\providecommand{\MRhref}[2]{%
  \href{http://www.ams.org/mathscinet-getitem?mr=#1}{#2}
}
\providecommand{\href}[2]{#2}
\begin{thebibliography}{HKMS09}

\bibitem[Ard19]{Ardila2019}
Rene Ardila, \emph{{Morita Equivalence of {$W^\ast$}-Correspondences and Their
  Hardy Algebras}}, Complex Anal. Oper. Theory \textbf{13} (2019), no.~5,
  2411--2441.

\bibitem[DP98a]{DPit98b}
Kenneth~R. Davidson and David~R. Pitts, \emph{{Nevanlinna-{P}ick interpolation
  for non-commutative analytic {T}oeplitz algebras}}, Integral Equations
  Operator Theory \textbf{31} (1998), no.~3, 321--337. \MR{1627901
  (2000g:47016)}

\bibitem[DP98b]{DPit98a}
\bysame, \emph{{The algebraic structure of non-commutative analytic {T}oeplitz
  algebras}}, Math. Ann. \textbf{311} (1998), no.~2, 275--303. \MR{1625750
  (2001c:47082)}

\bibitem[Har74]{Harris1974}
Lawrence~A. Harris, \emph{Bounded symmetric homogeneous domains in infinite
  dimensional spaces}, Proceedings on {I}nfinite {D}imensional {H}olomorphy
  ({I}nternat. {C}onf., {U}niv. {K}entucky, {L}exington, {K}y., 1973),
  Springer, Berlin, 1974, pp.~13--40. Lecture Notes in Math., Vol. 364.
  \MR{0407330}

\bibitem[HKM11a]{HKM2011b}
J.~William Helton, Igor Klep, and Scott McCullough, \emph{{Analytic mappings
  between noncommutative pencil balls}}, J. Math. Anal. Appl. \textbf{376}
  (2011), no.~2, 407--428. \MR{2747767 (2012b:46135)}

\bibitem[HKM11b]{HKM2011a}
\bysame, \emph{{Proper analytic free maps}}, J. Funct. Anal. \textbf{260}
  (2011), no.~5, 1476--1490. \MR{2749435}

\bibitem[HKMS09]{HKMS2009}
J.~William Helton, Igor Klep, Scott McCullough, and Nick Slinglend,
  \emph{{Noncommutative ball maps}}, J. Funct. Anal. \textbf{257} (2009),
  no.~1, 47--87. \MR{2523335 (2011b:47037)}

\bibitem[KVV09]{K-VV2009}
D.S. Kaliuzhnyi-Verbovetskyi and Victor Vinnikov, \emph{{Singularities of
  rational functions and minimal factorizations: the noncommutative and the
  commutative setting}}, Linear Algebra Appl. \textbf{430} (2009), no.~4,
  869--889. \MR{2489365 (2010a:47032)}

\bibitem[KVV12]{Kaliuzhnyi2012}
\bysame, \emph{{Noncommutative rational functions, their
  difference-differential calculus and realizations}}, Multidimensional Systems
  and Signal Processing \textbf{23} (2012), no.~1, 49--77.

\bibitem[KVV14]{kaliuzhnyi2014foundations}
\bysame, \emph{Foundations of free noncommutative function theory},
  Mathematical Surveys and Monographs, American Mathematical Society, 2014.

\bibitem[Mos04]{Mostafazadeh2004}
Ali. Mostafazadeh, \emph{{Pseudounitary operators and pseudounitary quantum
  dynamics}}, Journal of mathematical physics \textbf{45} (2004), no.~3,
  932--946.

\bibitem[MS04]{Muhly2004a}
Paul~S. Muhly and Baruch Solel, \emph{{Hardy algebras,
  {$W^\ast$}-correspondences and interpolation theory}}, Math. Ann.
  \textbf{330} (2004), no.~2, 353--415. \MR{2089431 (2006a:46073)}

\bibitem[MS08]{Muhly2008b}
\bysame, \emph{{Schur class operator functions and automorphisms of {H}ardy
  algebras}}, Doc. Math. \textbf{13} (2008), 365--411. \MR{2520475
  (2010g:46098)}

\bibitem[MS09]{Muhly2009}
\bysame, \emph{{The {P}oisson kernel for {H}ardy algebras}}, Complex Anal.
  Oper. Theory \textbf{3} (2009), no.~1, 221--242. \MR{2481905 (2011a:47166)}

\bibitem[MS11a]{Muhly2011b}
\bysame, \emph{{Morita transforms of tensor algebras}}, New York J. Math.
  \textbf{17A} (2011), 87--100. \MR{2782729}

\bibitem[MS11b]{Muhly2011a}
\bysame, \emph{{Representations of {H}ardy algebras: absolute continuity,
  intertwiners, and superharmonic operators}}, Integral Equations Operator
  Theory \textbf{70} (2011), no.~2, 151--203. \MR{2794388 (2012g:47218)}

\bibitem[MS13]{Muhly2013}
\bysame, \emph{{Tensorial {F}unction {T}heory: {F}rom {B}erezin {T}ransforms to
  {T}aylor's {T}aylor {S}eries and {B}ack}}, Integral Equations Operator Theory
  \textbf{76} (2013), no.~4, 463--508. \MR{3073943}

\bibitem[Pas73]{Paschke1973}
William~L. Paschke, \emph{Inner product modules over {$B^{\ast} $}-algebras},
  Trans. Amer. Math. Soc. \textbf{182} (1973), 443--468. \MR{0355613 (50
  \#8087)}

\bibitem[Pop91]{Popescu1991}
Gelu Popescu, \emph{{von {N}eumann inequality for {$(B(\mathcal{H})^n)_1$}}},
  Math. Scand. \textbf{68} (1991), no.~2, 292--304. \MR{1129595 (92k:47073)}

\bibitem[Tay72]{Tay72c}
Joseph~L. Taylor, \emph{{A general framework for a multi-operator functional
  calculus}}, Advances in Math. \textbf{9} (1972), 183--252. \MR{0328625 (48
  \#6967)}

\bibitem[Voi05]{Voi2005}
Dan Voiculescu, \emph{{Free probability and the von {N}eumann algebras of free
  groups}}, Rep. Math. Phys. \textbf{55} (2005), no.~1, 127--133. \MR{2126420
  (2005k:46180)}

\bibitem[Voi10]{Voi2010}
Dan-Virgil Voiculescu, \emph{{Free analysis questions {II}: the {G}rassmannian
  completion and the series expansions at the origin}}, J. Reine Angew. Math.
  \textbf{645} (2010), 155--236. \MR{2673426 (2012b:46144)}

\end{thebibliography}

\end{document}